\documentclass[12pt]{amsart}
\usepackage{amsmath}
\usepackage{amssymb}
\usepackage{amscd}
\usepackage{amsthm}
\usepackage{enumerate}
\usepackage{fullpage}
\usepackage{mathrsfs}
\usepackage{setspace}
\usepackage{mathtools}
\usepackage{xcolor}
\usepackage[all]{xy}
\usepackage{url}
\usepackage[colorlinks=true,linkcolor=blue,breaklinks,pagebackref]{hyperref}
\usepackage[lite,abbrev,alphabetic]{amsrefs}
\usepackage[poorman]{cleveref}
\usepackage{bbm}
\usepackage{tikz-cd}
\usepackage{graphicx}

\theoremstyle{definition}
\newtheorem{definition}{Definition}[section]

\theoremstyle{plain}
\newtheorem{theorem}[definition]{Theorem}
\newtheorem{proposition}[definition]{Proposition}
\newtheorem{lemma}[definition]{Lemma}

\newtheorem*{theorem*}{Theorem}

\theoremstyle{remark}
\newtheorem{remark}[definition]{Remark}

\crefname{theorem}{Theorem}{Theorems}
\crefname{proposition}{Proposition}{Propositions}
\crefname{lemma}{Lemma}{Lemmas}
\crefname{corollary}{Corollary}{Corollaries}
\crefname{conjecture}{Conjecture}{Conjectures}
\crefname{hypothesis}{Hypothesis}{Hypotheses}
\crefname{remark}{Remark}{Remarks}
\crefname{condition}{Condition}{Conditions}
\crefname{example}{Example}{Examples}

\def\N{\mathbb{N}}
\def\C{\mathbb{C}}
\def\R{\mathbb{R}}
\def\Q{\mathbb{Q}}
\def\Z{\mathbb{Z}}
\def\A{\mathbb{A}}

\def\GL{\mathrm{GL}}

\def\rd{\,\mathrm{d}}

\def\inf{\infty}

\def\bs{\backslash}

\DeclareMathOperator{\vol}{vol}

\DeclareMathOperator{\Ind}{Ind}

\DeclareMathOperator{\Ad}{Ad}

\DeclareMathOperator{\tr}{tr}

\newcommand{\fa}{\mathfrak{a}}

\newcommand{\fg}{\mathfrak{g}}

\newcommand{\fz}{\mathfrak{z}}

\newcommand{\fB}{\mathfrak{B}}

\newcommand{\cA}{\mathcal{A}}

\newcommand{\cC}{\mathcal{C}}

\newcommand{\cF}{\mathcal{F}}

\newcommand{\cL}{\mathcal{L}}
\newcommand{\cM}{\mathcal{M}}

\newcommand{\cP}{\mathcal{P}}

\newcommand{\cU}{\mathcal{U}}

\newcommand{\cX}{\mathcal{X}}

\newcommand{\reg}{\mathrm{reg}}

\newcommand{\sS}{\mathscr{S}}

\numberwithin{equation}{section}

\newcommand{\disc}{\mathrm{disc}}

\newcommand{\tG}{\widetilde{G}}
\newcommand{\tM}{\widetilde{M}}
\newcommand{\tL}{\widetilde{L}}
\newcommand{\tP}{\widetilde{P}}
\newcommand{\tQ}{\widetilde{Q}}
\newcommand{\tW}{\widetilde{W}}

\newcommand{\op}{\mathrm{op}}

\newcommand{\talpha}{\widetilde{\alpha}}
\newcommand{\tbeta}{\widetilde{\beta}}
\newcommand{\tlambda}{\widetilde{\lambda}}
\newcommand{\tLambda}{\widetilde{\Lambda}}
\newcommand{\tpi}{\widetilde{\pi}}

\newcommand{\tw}{\widetilde{w}}

\begin{document}

\title{On the absolute convergence of the spectral side of the twisted trace formula}

\author{Werner M\"uller}
\address{Universit\"at Bonn, Mathematisches Institut, Endenicher Allee 60, D-53115 Bonn, Germany}
\email{mueller@math.uni-bonn.de}

\author{Satoshi Wakatsuki}
\address{Faculty of Mathematics and Physics, Institute of Science and Engineering, Kanazawa University, Kakumamachi, Kanazawa, Ishikawa, 920-1192, Japan}
\email{wakatsuk@staff.kanazawa-u.ac.jp}

\date{\today}

\begin{abstract}
We establish the absolute convergence of the spectral side of the twisted trace formula for reductive algebraic groups. 
More precisely, we extend the absolute convergence theorem of the spectral side due to Finis--Lapid--M\"uller \cite{FLM11} to the twisted setting. 
This paper provides a detailed proof of the absolute convergence theorem, whose proof is known to experts but does not seem to have appeared in the literature.
The resulting formulas are motivated by applications to limit multiplicity problems and to Weyl laws for self-dual automorphic representations of $\GL_n$.
\end{abstract}

\maketitle

\tableofcontents

\section{Introduction}\label{sec:intro}

Arthur's trace formula is a fundamental tool in the theory of automorphic forms, and the absolute convergence theorem for its spectral side forms one of the foundations of the theory.
For the (untwisted) trace formula, Finis--Lapid--M\"uller \cite{FLM11} proved an absolute convergence theorem and established a refinement of the spectral side. 
The purpose of this paper is to extend their absolute convergence theorem to the twisted trace formula.
In the twisted setting, although no explanation is given, Labesse--Waldspurger remarked in \cite{LW13}*{Proof of Th\'eor\`eme 14.3.1} that the arguments used in the proof of the absolute convergence theorem of Finis--Lapid--M\"uller carry over to the twisted trace formula without modification.
Moreover, it is through this absolute convergence that they obtain a refinement of the spectral side; see \cite{LW13}*{\S14.3}.
Subsequently, Parab presented a proof of this assertion concerning the absolute convergence theorem in \cite{Par19}*{Proof of Theorem 7.2}.
However, that proof contains an error (see Remark \ref{rem:Parab} for details), and the form of the trace norm described in the proof does not yield a correct proof of the absolute convergence. 
Consequently, a proof of the absolute convergence theorem for the spectral side of the twisted trace formula appears to be unavailable in the existing literature.
In particular, the combinatorial formula (Theorem \ref{thm:main}) that is essential for applications does not appear in either \cite{LW13} or \cite{Par19}.

The absolute convergence argument of \cite{FLM11} makes it possible to enlarge the class of test functions admissible in the trace formula beyond those with compact support.
In particular, such test functions are required for Langlands' idea of Beyond Endoscopy in connection with functoriality. 
The same is true in the twisted setting. 
Indeed, Parab presented an application to a study of residues of Rankin--Selberg $L$-functions in \cite{Par19}*{\S8}. 
Furthermore, we expect that the detailed proof of the absolute convergence theorem established in this paper will facilitate future applications of the twisted trace formula, especially in situations where precise spectral estimates are required.
Indeed, the combinatorial formula \cite{FLM11}*{Theorem 2}, together with the formula obtained by combining it with the refinement of the spectral side, namely the second formula in \cite{FLM11}*{Corollary 1}, plays an important role in estimates of the spectral side, as illustrated by its applications to limit multiplicity formulas and Weyl laws in, for example, \cites{FLM15,FM21}. 
The combinatorial formula of the present paper (Theorem \ref{thm:main}), together with the corresponding formula obtained by combining it with the refinement of the spectral side (Theorem \ref{thm:conti}), is expected to play a similar role in the study of asymptotic problems involving the twisted trace formula.
For example, in \cite{TW26}, a conditional limit multiplicity formula for self-dual cuspidal representations of $\GL_n$ was established. 
Removing the condition imposed there requires a detailed analysis of the spectral side of the twisted trace formula, for which our combinatorial formula provides the necessary foundation.
Furthermore, it will be applied to the study of Weyl laws for self-dual cuspidal representations of $\GL_n$, which the authors are currently investigating.

The paper is organized as follows.
In \S\ref{sec:setup}, we review basic material, including the structure of disconnected reductive algebraic groups and their root systems.
In \S\ref{sec:Adja}, we discuss the adjacency relation for parabolic subsets, and in \S\ref{sec:tGtM}, we introduce $(\tG,\tM)$-families.
In \S\ref{sec:inducedfan}, we clarify the relationship between $(\tG,\tM)$-families and induced fans, which is needed for the proof of the combinatorial formula in the twisted setting.
In \S\ref{sec:measure}, we review the normalization of measures, and in \S\ref{sec:intertwining}, we recall the theory of intertwining operators, thereby preparing for the statement of the combinatorial formula.
We then present the combinatorial formula in the twisted setting in \S\ref{sec:combiatorial}.
Sections \S\ref{sec:rho} and \S\ref{sec:Frechet} provide the necessary preparation for the statement of the absolute convergence theorem, and in \S\ref{sec:absoconv} we prove the absolute convergence theorem in the twisted setting using the combinatorial formula.
After reviewing the modified kernel for the twisted trace formula in \S\ref{sec:kernel}, we conclude in \S\ref{sec:refinement} by describing the refinement of the spectral side combined with the combinatorial formula.

\medskip
{\em Acknowledgment.}
The second author is partially supported by JSPS Grant-in-Aid for Scientific Research (B) No. 23K20785. 

\section{Setup}\label{sec:setup}

In this section, we introduce the notation following the styles of Arthur \cites{Art88b,Art05}, Finis--Lapid--M\"uller \cite{FLM11}, and Labesse--Waldspurger \cite{LW13}.

We first fix the notation concerning number fields and adeles.
\begin{itemize}
    \item Let $F$ be an algebraic number field.
    \item $\A$ is the adele ring of $F$, $\A_f$ is the finite adele ring of $F$, and $F_\inf \coloneqq \R\otimes_\Q F$. 
    \item For each place $v$ of $F$, we denote by $F_v$ the completion of $F$ at $v$, and $|\;|_v$ the normalized valuation on $F_v$. 
    We define the idele norm by $|\;|\coloneqq\prod_v |\;|_v$. 
\end{itemize}

We describe the setting of the algebraic groups.
For details on algebraic groups, we refer the reader to \cites{Bor91,Spr98}.
\begin{itemize}
    \item Let $G'$ be a reductive algebraic group over $F$. 
    \item We denote by $G$ the connected component of $1$ in $G'$. It is known that $G$ is a normal closed subgroup of $G'$, of finite index, and defined over $F$. 
    \item Take a connected component $\tG$ of $G'$. Throughout this paper, we assume that the set of $F$-rational points $\tG(F)$ is nonempty.
    \item Fix a minimal parabolic subgroup $P_0$ of $G$ over $F$. 
    \item Fix a Levi subgroup $M_0$ of $P_0$ over $F$. 
    \item Define a homomorphism $\Ad$ from $G'$ to the automorphism group of $G$ by $\Ad(g)h \coloneqq ghg^{-1}$ for $g \in G'$ and $h \in G$.
    \item  There exists an element $\delta\in \tG(F)$ such that $\Ad(\delta)(P_0)=P_0$ and $\Ad(\delta)(M_0)=M_0$.
    Note that such a $\delta \in \tG(F)$ is unique up to multiplication by elements of $M_0$. 
    In the following, we fix one such $\delta \in \tG(F)$ and set $\theta \coloneqq \Ad(\delta)$.
    Then $\theta$ is an automorphism of $G$ defined over $F$, and $P_0$ and $M_0$ are $\theta$-stable, 
    that is, $\theta(P_0)=P_0$ and $\theta(M_0)=M_0$.
\end{itemize}
\begin{remark}
    We discuss the twisted trace formula in the framework of reductive algebraic groups, including the non-connected case. 
    This is the same setting as in Arthur's work.
    At present, however, the twisted trace formula is formulated in the more general setting of twisted spaces; see \cite{LW13}*{Ch.~2}. 
    For our purposes, the framework of reductive algebraic groups is sufficient, and in order to simplify the exposition and make its relation to Arthur's formulation more transparent, we work in this setting throughout the paper.
\end{remark}

Let us introduce the notation for Levi subgroups.
\begin{itemize}
    \item Let $T_0$ denote the maximal $F$-split torus in $M_0$, that is, $M_0$ is the centralizer of $T_0$ in $G$. 
    \item $\cL^G(M_0)$ is the set of Levi $F$-subgroups containing $M_0$, that is, the finite set of centralizers of $F$-subtori of $T_0$. 
\item For $M\in\cL^G(M_0)$, we use the following notation:
\begin{itemize}        
    \item $T_M$ is the $F$-split part of the identity component of the center of $M$. 
    \item $\cL^G(M)$ is the set of Levi $F$-subgroups containing $M$. 
    \item $\cP^G(M)$ is the set of parabolic $F$-subgroups with Levi part $M$. 
    \item $\cF^G(M)\coloneqq\coprod_{L\in\cL^G(M)} \cP^G(L)$ is the finite set of parabolic $F$-subgroups of $G$ containing $M$. 
    \item For each $P\in\cP^G(M)$, let $N_P$ denote the unipotent radical of $P$, and let $M_P$ denote the Levi subgroup of $P$ such that $M_P\in \cL^G(M_0)$. 
\end{itemize}
\end{itemize}
Next, we explain the parabolic subsets of $\tG$ and related notions.
\begin{itemize}
    \item Take a parabolic subgroup $P\in\cF^G(M_0)$. Let $\tP$ denote the normalizer of $P$ in $\tG$. If $\tP$ is nonempty, we call $\tP$ a parabolic subset of $\tG$. In this case, note that $\tP(F)$ is also nonempty.
    \item Set $\tM_0\coloneqq M_0\delta$ and $\cF^{\tG}(\tM_0)\coloneqq\{ \tP \mid P\in\cF^G(M_0)$ and $\tP\supset\tM_0 \}$. Note that $\tP\supset\tM_0$ implies $\tP\neq \emptyset$. 
    \item Any parabolic subset $\tP\in\cF^{\tG}(\tM_0)$ satisfies $\tP=P\delta$ and $\theta(P)=P$. 
    On the other hand, if $P \in \cF^G(M_0)$ is $\theta$-stable, then we have $\tP \in \cF^{\tG}(\tM_0)$.
    \item Take a parabolic subset $\tP\in \cF^{\tG}(\tM_0)$, that is, $P\in\cF^G(M_0)$ and $\theta(P)=P$. 
    Then, we have $\theta(M_P)=M_P$ and $\theta(N_P)=N_P$. Hence, we have a Levi decomposition $\tP = \tM_P N_P$, where $\tM_P\coloneqq M_P\, \delta$. 
    We call this $\tM_P$ a Levi subset of $\tG$.
\end{itemize}
For a Levi subset $\tM$ of $\tG$, we use the following notation:
\begin{itemize}
    \item $T_{\tM}$ is the $F$-split part of the identity component of the centralizer of $\tM$ in $M$. 
    Then, $T_{\tM}$ is the identity component of $\{ t\in T_M \mid \theta(t)=t \}$.
    \item $\cF^{\tG}(\tM)\coloneqq \{ \tP\in\cF^{\tG}(\tM_0) \mid \tM_P\supset \tM\}$. 
    \item $\cL^{\tG}(\tM)\coloneqq \{ \tM_P \mid \tP\in \cF^{\tG}(\tM)\}$. 
    \item $\cP^{\tG}(\tM) \coloneqq \{ \tP\in \cF^{\tG}(\tM) \mid \tM_P=\tM \}$.
\end{itemize}
For a Levi subgroup $M\in\cL^G(M_0)$, we introduce the notation related to roots.
\begin{itemize}
    \item Denote by $X^*(M)$ (resp. $X^*(T_M)$) the lattice  of $F$-rational characters of $M$ (resp. $T_M$). 
    Note that, by restriction, $X^*(M)$ is a sublattice of $X^*(T_M)$. 
    \item $\fa_M^*$ is the $\R$-vector space spanned by $X^*(M)$. Set $\fa^*_{M,\C}\coloneqq\fa_M^*\otimes_\R \C$. 
    \item $\fa_M$ is the dual space of $\fa_M^*$, which is spanned by the co-characters of $T_M$. 
    \item $H_M\colon M(\A)\to\fa_M$ is the homomorphism given by $\langle \chi,H_M(m) \rangle=\log |\chi(m)|$, $\chi\in X^*(M)$, $m\in M(\A)$. 
    \item $\Sigma_M$ is the set of reduced roots of $T_M$ on the Lie algebra of $G$. 
    \item For any $\alpha\in \Sigma_M\, (\subset\fa_M^*)$, we write $\alpha^\vee\in\fa_M$ for the corresponding coroot of $\alpha$. Set $\Sigma_M^\vee \coloneqq \{\alpha^\vee \mid \alpha\in\Sigma_M\}$. 
    \item Take a Levi subgroup $L\in\cL^G(M)$.
    We identify $\fa_L^*$ (resp. $\fa_L$) with a subspace of $\fa_M^*$ (resp. $\fa_M$) by the restriction homomorphism $X^*(L)\to X^*(M)$ (resp. $X^*(T_M)\to X^*(T_L)$). 
    Denote by $\fa_M^L$ (resp. $(\fa_M^L)^*$) the annihilator of $\fa_L^*$ in $\fa_M$ (resp. $\fa_L$ in $\fa_M^*$).
    Then, $\fa_M=\fa_L\oplus\fa_M^L$ and $\fa_M^*=\fa_L^*\oplus(\fa_M^L)^*$. 
\end{itemize}
For a parabolic subgroup $P\in\cP^G(M)$, we introduce the notation related to positive roots.
\begin{itemize}
    \item $\fa_P\coloneqq\fa_M$. 
    \item $\Sigma_P$ is the set of reduced roots of $T_M$ on the Lie algebra of $N_P$.
    \item $\Delta_P$ is the set of simple roots of $P$, which is a basis of $(\fa_P^G)^*$. 
    \item $\Sigma_P^\vee\coloneqq\{ \alpha^\vee \mid \alpha\in\Sigma_P\}$ and $\Delta_P^\vee\coloneqq\{ \alpha^\vee \mid \alpha\in\Delta_P\}$. 
    \item For $Q\in\cF^G(M)$, $Q\supset P$, we set $\Delta_P^Q\coloneqq \{ \alpha\in \Delta_P \mid \alpha|_{\fa_Q}=0   \}$ and $(\Delta_P^Q)^\vee\coloneqq\{ \alpha^\vee \mid \alpha\in\Delta_P^Q\}$. 
    The following two properties are known.
    \begin{itemize}
        \item The map $Q \longmapsto \Delta_{P}^{Q}$ gives a bijection from $\{ Q\in \cF^G(M) \mid Q \supset P \,\}$ onto the set of subsets of $\Delta_{P}$.
        \item $\Delta_{P}^{Q}$ (resp. $(\Delta_{P}^{Q})^\vee$) forms a basis of $(\fa_{P}^{Q})^*$ (resp. $\fa_{P}^{Q}$).
    \end{itemize}
    \item $P^\op$ is the opposite of $P$, that is, $\Sigma_{P^\op}=-\Sigma_P$ and $\Delta_{P^\op}=-\Delta_P$. Then, $\Sigma_M=\Sigma_P\sqcup\Sigma_{P^\op}$.
\end{itemize}
For $\tM\in\cL^{\tG}(\tM_0)$ and $\tP\in\cP^{\tG}(\tM)$, 
we introduce the notation related to roots.
Note that $M$, $T_M$, and $N_P$ are $\theta$-stable. 
\begin{itemize}
    \item $\theta$ acts on $X^*(T_M)$ as $\theta(\chi)(m)\coloneqq \chi(\theta^{-1}(m))$, $\chi\in X^*(T_M)$, $m\in T_M$, and $\theta$ also acts on $X_*(T_M)$ as $(\theta(\eta))(a)\coloneqq \theta(\eta(a))$, $\eta\in X_*(T_M)$, $a\in \mathbb{G}_m$. 
    Then, we define
    \[
    \fa^*_{\tM}=\fa_{\tP}^*\coloneqq \{ \lambda\in \fa^*_M \mid \theta(\lambda)=\lambda \}, \qquad 
    \fa_{\tM}=\fa_{\tP}\coloneqq \{ H\in \fa_M \mid \theta(H)=H \}.
    \]
    \item For $\tL\in \cL^{\tG}(\tM)$ and $\tQ\in\cP^{\tG}(\tL)$, $Q\supset P$, we obtain
    \[
    \fa_{\tM}^*=\fa_{\tL}^*\oplus (\fa^{\tL}_{\tM})^* \quad (\text{equivalently }\fa_{\tP}^*=\fa_{\tQ}^*\oplus (\fa^{\tQ}_{\tP})^*),
    \]
    \[
    \fa_{\tM}=\fa_{\tL}\oplus \fa^{\tL}_{\tM} \quad (\text{equivalently }\fa_{\tP}=\fa_{\tQ}\oplus \fa^{\tQ}_{\tP}),
    \]
    where $(\fa^{\tL}_{\tM})^*=(\fa^{\tQ}_{\tP})^*\coloneqq\{\lambda\in (\fa^{L}_{M})^* \mid \theta(\lambda)=\lambda\}$ and $\fa^{\tL}_{\tM}=\fa^{\tQ}_{\tP}\coloneqq\{H\in \fa^{L}_{M} \mid \theta(H)=H\}$.
    \item $\theta$ acts bijectively on $\Sigma_P$ and on $\Delta_P$, since $N_P$ is $\theta$-stable.
    Choose a positive integer $\ell$ such that $\theta^\ell$ acts trivially on $\Delta_P$.
    \item For $\alpha\in\Sigma_M$, we set
    \[
    \talpha\coloneqq\frac{1}{\ell}\sum_{j=0}^{\ell-1}\theta^j(\alpha), \quad \talpha^\vee\coloneqq\frac{1}{\ell}\sum_{j=0}^{\ell-1}\theta^j(\alpha^\vee).
    \]
    Let $\Sigma_{\tM}$ be obtained from $\{\talpha \mid \alpha\in\Sigma_M\}$ by retaining the shortest root from each positive ray. 
    In addition, 
    \begin{align*}
        &    \Sigma_{\tP}\coloneqq \Sigma_{\tM}\cap \{\talpha\mid \alpha\in\Sigma_P\}, \quad \Delta_{\tP}\coloneqq \{\talpha\mid \alpha\in\Delta_P\} , \\ 
        &     \Sigma_{\tM}^\vee\coloneqq\{\talpha^\vee\mid \alpha\in\Sigma_M , \; \;  \talpha\in\Sigma_{\tM}\}, \quad \Sigma_{\tP}^\vee\coloneqq\Sigma_{\tM}^\vee\cap \{\talpha^\vee\mid \alpha\in\Sigma_P\}, \quad \Delta_{\tP}^\vee\coloneqq\{\talpha^\vee\mid \alpha\in\Delta_P\}. 
    \end{align*}
    \item $\Sigma_{\tM}$ becomes the set of reduced roots of $T_{\tM}$ on the Lie algebra of $G$.
    \item $\widetilde{P^\op}= P^\op \delta$ also belongs to $\cP^{\tG}(\tM)$, hence $\Sigma_{\tM}=\Sigma_{\tP}\sqcup\Sigma_{\tP^\op}$. 
    \item For $\tQ\in\cP^{\tG}(\tL)$, $Q\supset P$, we set
    \[
    \Delta_{\tP}^{\tQ}\coloneqq \{\talpha\in\Delta_{\tP} \mid \talpha|_{\fa_{\tQ}}=0 \}, \quad (\Delta_{\tP}^{\tQ})^\vee\coloneqq \{\talpha^\vee  \mid \talpha\in \Delta_{\tP}^{\tQ} \}.
    \]
\end{itemize}

We record the following two basic facts, cf. \cite{LW13}*{Lemmes 2.7.1, 2.7.2}.
\begin{itemize}
    \item Fix $\tP \in \cF^{\tG}(\tM_0)$.
    The map $\tQ \longmapsto \Delta_{\tP}^{\tQ}$ gives a bijection from the set $\{\, \tQ \in \cF^{\tG}(\tM_0) \mid \tQ \supset \tP \,\}$ onto the set of subsets of $\Delta_{\tP}$.
    \item Let $\tP, \tQ \in \cF^{\tG}(\tM_0)$ and assume $\tP \subset \tQ$. Then $\Delta_{\tP}^{\tQ}$ (resp. $(\Delta_{\tP}^{\tQ})^\vee$) forms a basis of $(\fa_{\tP}^{\tQ})^*$ (resp. $\fa_{\tP}^{\tQ}$).
\end{itemize}

\section{Adjacency}\label{sec:Adja}

We first fix the notation related to adjacency.
\begin{itemize}
    \item For $j\in\Z_{\ge 0}$, set
    \[
    \cL^G_j(M)\coloneqq\{ L\in\cL^G(M) \mid \dim\fa_M^L=j \}, \quad \cF^G_j(M)\coloneqq \bigsqcup_{L\in\cL_j^G(M)} \cP^G(L),
    \]
    \[
    \cL^{\tG}_j(\tM)\coloneqq\{ \tL\in\cL^{\tG}(\tM) \mid \dim\fa_{\tM}^{\tL}=j \}, \quad \cF^{\tG}_j(\tM)\coloneqq \bigsqcup_{\tL\in\cL_j^{\tG}(\tM)} \cP^{\tG}(\tL).
    \]
    \item For $P$, $Q\in\cP^G(M)$, we write $\overline{PQ}$ for the closure of $PQ$. 
    \item For $\alpha\in\Sigma_M$, we set $H_\alpha\coloneqq\{ \lambda\in\fa_M^* \mid \langle \lambda, \alpha^\vee\rangle=0 \}$.
    \item For $\talpha\in\Sigma_{\tM}$, we set $H_{\talpha}\coloneqq\{ \tlambda\in\fa_{\tM}^* \mid \langle \tlambda, \talpha^\vee\rangle=0 \}$.
    \item $\fa^*_{P,+}$ (resp. $\fa^*_{\tP,+}$) is the closure of the Weyl chamber of $P$ (resp. $\tP$), that is,
    \[
    \fa^*_{P,+} = \{ \lambda\in\fa_M^* \mid \langle \lambda,\alpha^\vee\rangle\ge0 \text{ for all } \alpha\in\Delta_P\}.
    \]
    \[
    \text{(resp. } \fa^*_{\tP,+} = \{ \tlambda\in\fa_{\tM}^* \mid \langle \tlambda,\talpha^\vee\rangle\ge0 \text{ for all } \talpha\in\Delta_{\tP}\}.)
    \]
    \item $\Sigma(G,M)$ (resp. $\Sigma(\tG,\tM)$) is the polyhedral fan whose chambers are the closures of the connected components of $\fa_M^*\setminus \cup_{\alpha\in\Sigma_M}H_\alpha$ (resp. $\fa_{\tM}^*\setminus \cup_{\talpha\in\Sigma_{\tM}}H_{\talpha}$). 
    For polyhedral fans arising from hyperplane arrangements, see, for example, \cite{Zie95}*{Ch.7}.
    \item Let $\Sigma_j(G,M)$ (resp. $\Sigma_j(\tG,\tM)$) denote the subset of $\Sigma(G,M)$ (resp. $\Sigma(\tG,\tM)$) consisting of cones of codimension $j\in\Z_{\ge 0}$.
\end{itemize}

In the following, we briefly recall the notion of adjacency for parabolic subgroups and its basic properties.
Let $M \in \cL^G(M_0)$, $\alpha\in\Sigma_M$, and $P$, $Q\in\cP^G(M)$.
\begin{itemize}
    \item If $\Sigma_P\cap\Sigma_{Q^\op}=\{\alpha\}$, we say that $P$ and $Q$ are adjacent along $\alpha$, and we write $P\mid^\alpha Q$. 
    \item If $\alpha\in\Sigma_P$ and $H_\alpha$ is spanned by the wall $\fa^*_{P,+}\cap \fa^*_{Q,+}$, we say that $\fa^*_{P,+}$ and $\fa^*_{Q,+}$ are adjacent along $H_\alpha$. 
    \item The following three conditions are equivalent. 
    \begin{itemize}
        \item $P\mid^\alpha Q$.
        \item $\fa^*_{P,+}$ and $\fa^*_{Q,+}$ are adjacent along $H_\alpha$. 
        \item $\Delta_{P}^{\overline{PQ}}=\{\alpha\}$. (This implies $\overline{PQ}\in\cF^G_1(M)$.) 
    \end{itemize}
    \item The mapping $\cF^G(M)\ni P\mapsto \fa^*_{P,+}\in\Sigma(G,M)$ is bijective. 
Under this bijection, $\cF_j^G(M)$ corresponds to $\Sigma_j(G,M)$. 
    \item For two given parabolic subgroups $Q_1$, $Q_2\in\cP^G(M)$, there exist $R_0$, $R_1,\dots,R_k\in\cP^G(M)$ and $\alpha_1,\dots,\alpha_k\in\Sigma_M$ such that
\[
Q_1=R_0\mid^{\alpha_1}R_1\mid^{\alpha_2}R_2 \cdots R_{k-1}\mid^{\alpha_k} R_k=Q_2.
\]
\end{itemize}

Let $\tM \in \cL^{\tG}(\tM_0)$, $\talpha\in\Sigma_{\tM}$, and $\tP$, $\tQ\in\cP^{\tG}(\tM)$. 
If $\Sigma_{\tP}\cap\Sigma_{\tQ^\op}=\{\talpha\}$, we say that $\tP$ and $\tQ$ are adjacent along $\talpha$, and we write $\tP\mid^{\talpha} \tQ$. 
If $\talpha\in\Sigma_{\tP}$ and $H_{\talpha}$ is spanned by the wall $\fa^*_{\tP,+}\cap \fa^*_{\tQ,+}$, we say that $\fa^*_{\tP,+}$ and $\fa^*_{\tQ,+}$ are adjacent along $H_{\talpha}$. 
\begin{lemma}\label{lem:troot}
    The following three conditions are equivalent.
    \begin{itemize}
        \item[{\rm(i)}] $\tP \mid^{\talpha} \tQ$.
        \item[{\rm(ii)}] $\fa^*_{\tP,+}$ and $\fa^*_{\tQ,+}$ are adjacent along $H_{\talpha}$.
        \item[{\rm(iii)}] $\Delta_{\tP}^{\widetilde{\overline{PQ}}}=\{\talpha\}$. (This implies $\widetilde{\overline{PQ}}\in\cF^{\tG}_1(\tM)$.)
    \end{itemize}
\end{lemma}
\begin{proof}
Since (i) is equivalent to $\Sigma_{\tP}\setminus\{\talpha\}=\Sigma_{\tQ}\setminus\{-\talpha\}$ 
and (ii) is equivalent to $\fa^*_{\tP,+}\cap \fa^*_{\tQ,+}=\{\tlambda\in \fa^*_{\tP,+} \mid \langle \tlambda, \talpha^\vee \rangle=0\}$, we obtain the equivalence of \textup{(i)} and \textup{(ii)}.
Hence, it is sufficient to prove that \textup{(i)} and \textup{(iii)} are equivalent. 
In the proof below, we note that $P, Q \in \cP^G(M)$ are $\theta$-stable.
For $\tbeta\in\Sigma_{\tM}$, we define the root space $\fg_{\tbeta}$ associated with $\tbeta$ by
\[
\fg_{\tbeta}\coloneqq \bigoplus_{m\ge 1} \{X\in\fg\mid  \forall t\in T_{\tM}, \;\; \Ad(t)X=\tbeta(t)^m X \},
\]
where $\fg$ denotes the Lie algebra of $G$ and we identify $\tbeta$ with an element of $X^*(T_{\tM})$.


We first show that (i) implies (iii).
Suppose $\tP\mid^{\talpha}\tQ$. 
Take an element $\tbeta\in \Sigma_{\tP}\cap\Sigma_{\tQ}$. 
Then, $\fg_{\tbeta}$ is a subspace of $\mathrm{Lie}(N_{\overline{PQ}})$, 
hence $\tbeta|_{\fa_{\widetilde{\overline{PQ}}}}\neq 0$.
On the other hand, it follows from $\tP\mid^{\talpha}\tQ$ that $\alpha\in\Sigma_{\tP}$ and $-\alpha\in\Sigma_{\tQ}$, hence $\talpha|_{\fa_{\widetilde{\overline{PQ}}}}=0$.
Since $\Sigma_{\tP}\cap\Sigma_{\tQ}=\Sigma_{\tP}\setminus\{\talpha\}$, $\talpha$ is the unique element of $\Sigma_{\tP}$ which vanishes on $\fa_{\widetilde{\overline{PQ}}}$. Thus, we obtain
$\Delta_{\tP}^{\widetilde{\overline{PQ}}}=\{\talpha\}$.

Next we show that (iii) implies (i).
Suppose $\Delta_{\tP}^{\widetilde{\overline{PQ}}}=\{\talpha\}$.
Then, $(\fa_{\tP}^{\widetilde{\overline{PQ}}})^*$ is generated by $\talpha$. 
This means that, if $\tbeta\in \Sigma_{\tP}\setminus\{\talpha\}$ or $\tbeta\in \Sigma_{\tQ}\setminus\{-\talpha\}$, then $\tbeta|_{\fa_{\widetilde{\overline{PQ}}}}\neq 0$. Hence $\fg_{\tbeta}$ is a subspace of $\mathrm{Lie}(N_{\overline{PQ}})$.
From this we obtain
\[
\mathrm{Lie}(N_{\overline{PQ}})=\bigoplus_{\tbeta\in\Sigma_{\tP}\setminus\{\talpha\}} \fg_{\tbeta}=\bigoplus_{\tbeta\in\Sigma_{\tQ}\setminus\{-\talpha\}} \fg_{\tbeta}.
\]
Therefore, $\Sigma_{\tP}\setminus\{\talpha\}=\Sigma_{\tQ}\setminus\{-\talpha\}$, that is, $\tP\mid^{\talpha}\tQ$.
\end{proof}

As in the usual case, the following correspondence holds.
\begin{lemma}\label{lem:corrFSigma}
    The mapping $\cF^{\tG}(\tM)\ni \tP\mapsto \fa^*_{\tP,+}\in\Sigma(\tG,\tM)$ is bijective. Under this bijection, $\cF_j^{\tG}(\tM)$ corresponds to $\Sigma_j(\tG,\tM)$. 
\end{lemma}
\begin{proof}
    Since the centralizer of a $\theta$-stable $F$-split subtorus in $T_{\tM}$ is a $\theta$-stable Levi subgroup in $\cL^G(M)$ (see, e.g., \cite{Bor91}*{20.4 Proposition}), the map $\tL \longmapsto \fa_{\tL}^*$ from $\cL^{\tG}(\tM)$ to the set of subspaces generated by cones in $\Sigma(\tG,\tM)$ is bijective.    
    Therefore, it suffices to show that the map from $\cF^{\tG}_0(\tM)$ to $\Sigma_0(\tG,\tM)$ is bijective. 
    
    Since the injectivity of this map is clear, it remains to prove that it is surjective. 
    For a cone $C\in \Sigma_0(\tG,\tM)$, define
\[
\Sigma_C\coloneqq\{\, \talpha\in\Sigma_{\tM}\mid 
\langle\lambda,\talpha^\vee\rangle\ge 0 \text{ for all } \lambda\in C \,\}.
\]
Let us recall the notation $\fg_{\talpha}$ defined in the proof of Lemma \ref{lem:troot}.
Let $\fg_C$ be the Lie algebra generated by $\fg_{\talpha}$, $\talpha\in\Sigma_C$.
Since the codimension of $C$ is $0$, we have $\Sigma_C\sqcup(-\Sigma_C)=\Sigma_{\tM}$.
Hence $\Sigma_C$ forms a system of positive roots in $\Sigma_{\tM}$, so there exists a parabolic subset $\tP \in \cP^{\tG}(\tM)$ such that $\fg_C$ is the Lie algebra of $N_P$.
Thus, we have $\fa^*_{\tP,+}=C$, and the surjectivity is proved.
\end{proof}

\section{\texorpdfstring{$(\tG,\tM)$}{GM}-family}\label{sec:tGtM}

Fix $M\in\cL^G(M_0)$, and let us first recall the notion of a $(G,M)$-family.
A family of smooth functions $\{c_P(\lambda)\mid P\in\cP^G(M)\}$ 
$(\lambda\in \fa_M^*)$ on $\fa_M^*$ is called a $(G,M)$-family if,
whenever $P, Q\in\cP^G(M)$ are adjacent along $\alpha\in\Sigma_M$, 
we have $c_P(\lambda)=c_Q(\lambda)$ for $\lambda\in H_\alpha$.
The following assertion is a fact stated in \cite{Art82}*{p.1297}.
\begin{lemma}\label{lem:GM}
Let $\{c_P(\lambda)\mid P\in\cP^G(M)\}$ $(\lambda\in \fa_M^*)$ be a $(G,M)$-family.
For $L\in \cL^G(M)$ and $Q\in \cP^G(L)$, take $P\in\cP^G(M)$ such that $P\subset Q$, and define $c_Q(\lambda_L)\coloneqq c_P(\lambda_L)$ $(\lambda_L\in\fa_L^*)$.
Then, $c_Q(\lambda_L)$ does not depend on the choice of $P$, and $\{c_Q(\lambda_L)\mid Q\in\cP^G(L)\}$ $(\lambda_L\in \fa_L^*)$ forms a $(G,L)$-family.
\end{lemma}
\begin{proof}    
It is sufficient to prove that, for any $P_1$, $P_2\in\cP^G(M)$ satisfying $P_1\subset Q$ and $P_2\subset Q$, we have $c_{P_1}(\lambda_L)=c_{P_2}(\lambda_L)$, $\lambda_L\in \fa_L^*$. 
The parabolic subgroups $P_1\cap M_Q$ and $P_2\cap M_Q$ can be connected by a sequence of adjacent parabolic subgroups in $\cP^{M_Q}(M)$.
Moreover, the map
$\{P\in \cP^G(M)\mid P\subset Q\}\ni P\mapsto P\cap M_Q\in \cP^{M_Q}(M)$
is a bijection.
Hence $P_1$ and $P_2$ are connected by a sequence of adjacent parabolic subgroups in $\{P\in \cP^G(M)\mid P\subset Q\}$.
Therefore, the proof is completed.
\end{proof}

Choose a Levi subset $\tM\in \cL^{\tG}(\tM_0)$, that is, $M\in \cL^G(M_0)$ and $\theta(M)=M$. 
A family of smooth functions $\{c_{\tP}(\tlambda)\mid \tP\in\cP^{\tG}(\tM)\}$ 
$(\tlambda\in \fa_{\tM}^*)$ on $\fa_{\tM}^*$ is called a $(\tG,\tM)$-family if,
whenever $\tP, \tQ\in\cP^{\tG}(\tM)$ are adjacent along $\talpha\in\Sigma_{\tM}$, 
we have $c_{\tP}(\tlambda)=c_{\tQ}(\tlambda)$ for $\tlambda\in H_{\talpha}$.
The following assertion is a fact stated in \cite{LW13}*{p.47}.
\begin{lemma}\label{lem:tGtM}
Let $\{d_P(\lambda)\mid P\in\cP^G(M)\}$ $(\lambda\in \fa_M^*)$ be a $(G,M)$-family.
For $\tP\in\cP^{\tG}(\tM)$ and $\tlambda\in \fa_{\tM}^*$, define $c_{\tP}(\tlambda)\coloneqq d_P(\tlambda)$.
Then $\{c_{\tP}(\tlambda)\mid \tP\in\cP^{\tG}(\tM)\}$ is a $(\tG,\tM)$-family.
\end{lemma}
\begin{proof}
Assume that $\tP, \tQ \in \cP^{\tG}(\tM)$ and $\tP \mid^{\talpha} \tQ$. 
It suffices to prove the inclusion $H_{\talpha} \subset \fa_{\widetilde{\overline{PQ}}}^*$.
Indeed, by Lemma \ref{lem:GM}, we have $d_P(\lambda)=d_Q(\lambda)$ for every $\lambda\in\fa_{\overline{PQ}}^*$.

Let $\lambda \in H_{\talpha}$. 
Then $\langle \lambda, \talpha^\vee \rangle = 0$. 
By Lemma \ref{lem:troot} (iii), we have $\Delta_{\tP}^{\widetilde{\overline{PQ}}} = \{\talpha \}$, and $\talpha^\vee$ forms a basis of $\fa_{\tP}^{\widetilde{\overline{PQ}}}$. 
Hence, $\lambda$ is orthogonal to $\fa_{\tP}^{\widetilde{\overline{PQ}}}$, that is, $\lambda \in \fa_{\widetilde{\overline{PQ}}}^*$. 
Therefore, we obtain $H_{\talpha} \subset \fa_{\widetilde{\overline{PQ}}}^*$.
\end{proof}

From the above two lemmas, we obtain the following statement.
\begin{proposition}\label{prop:GM}
    Let $M\in\cL^G(M_0)$ be a Levi subgroup, and let $\tL\in \cL^{\tG}(\tM_0)$ be a Levi subset satisfying $L\supset M$. 
    Let $\{d_P(\lambda)\mid P\in\cP^G(M)\}$ $(\lambda\in \fa_M^*)$ be a $(G,M)$-family, and for $\tQ\in\cP^{\tG}(\tL)$ and $\tlambda_L\in \fa_{\tL}^*$, define $c_{\tQ}(\tlambda_L)\coloneqq d_Q(\tlambda_L)$.
    Then $\{c_{\tQ}(\tlambda_L)\mid \tQ\in\cP^{\tG}(\tL)\}$ $(\tlambda_L\in \fa_{\tL}^*)$ is a $(\tG,\tL)$-family.
\end{proposition}
\begin{proof}
    This statement follows from Lemmas \ref{lem:GM} and \ref{lem:tGtM}.
\end{proof}

\section{Induced fan}\label{sec:inducedfan}

For an arbitrary vector subspace $\mathfrak{U}$ of $\fa_M$, we define the polyhedral fan $\Sigma(G,M,\mathfrak{U})^\sharp$ on $\mathfrak{U}^\perp$ by
\[
\Sigma(G,M,\mathfrak{U})^\sharp\coloneqq \{ \cC\cap \mathfrak{U}^\perp \mid \cC\in\Sigma(G,M) \}.
\]
The fan $\Sigma(G,M,\mathfrak{U})^\sharp$ is called the induced fan.
For details, see \cite{FL11}*{\S7}.

The following statement concerns induced fans used in the discussion of \cite{FLM11}*{\S2.4}.
\begin{lemma}\label{lem:indfun1}
    Let $M\in\cL^G(M_0)$ and $L\in\cL^G(M)$. 
    Then one has $\Sigma(G,L)=\Sigma(G,M,\fa_M^L)^\sharp$.
\end{lemma}
\begin{proof}
For $P$, $Q\in\cF^G(M)$ with $P\subset Q$, we have $\fa_{P,+}^*\cap \fa_Q^*=\fa_{Q,+}^*$ by the definitions. 
This implies the assertion.
\end{proof}

In the following, we consider the induced fan required for the twisted version.
Fix $\tM\in\cL^{\tG}(\tM_0)$ once and for all. 
Here $M\in\cL^G(M_0)$ and $\theta(M)=M$. 
Put $\fa_M^{1-\theta}\coloneqq\langle H-\theta(H)\mid H\in\fa_M\rangle$. 
Then we have
\[
\fa_M=\fa_{\tM}\oplus \fa_M^{1-\theta}.
\]
The space $\fa_{\tM}^*$ is the annihilator of $\fa_M^{1-\theta}$ in $\fa_M^*$. 
\begin{lemma}\label{lem:indfun2}
Under the above setting, we have $\Sigma(\tG,\tM)=\Sigma(G,M,\fa_M^{1-\theta})^\sharp$.
\end{lemma}
\begin{proof}
Since $M$ is $\theta$-stable, $\theta$ acts on $\Sigma_M$. 
Hence, for $\lambda\in\fa_{\tM}^*$, we have
\begin{equation*}\label{eq:invariant}
\langle \lambda,\alpha^\vee\rangle=\langle\lambda, \theta^m(\alpha)^\vee\rangle,\qquad \forall\alpha\in\Delta_P,\quad \forall m\in\Z .  
\end{equation*}
It follows from this equality that, for $\tP\in\cF^{\tG}(\tM)$ we have $\fa_{P,+}^*\cap \fa_{\tM}^*=\fa_{\tP,+}^*$. 
This implies that the cones $\fa_{\tP,+}^*$, $\tP\in\cF^{\tG}(\tM)$ are cones of the induced fan $\Sigma(G,M,\fa_M^{1-\theta})^\sharp$. 
By Lemma \ref{lem:corrFSigma}, the cones $\fa_{\tP,+}^*$, $\tP\in \cF^{\tG}(\tM)$, form the polyhedral fan $\Sigma(\tG,\tM)$. 
Therefore, since $\Sigma(\tG,\tM)$ is a complete polyhedral fan on $\fa_{\tM}^*$, no additional cone can appear in $\Sigma(G,M,\fa_M^{1-\theta})^\sharp$. 
Hence the two polyhedral fans coincide.
\end{proof}

From the above two lemmas, we obtain the following statement.
\begin{proposition}\label{prop:indfun}
    Let $M\in\cL^G(M_0)$ be a Levi subgroup, and let $\tL\in \cL^{\tG}(\tM_0)$ be a Levi subset satisfying $L\supset M$. 
    Under this setting, note that we have the direct sum decomposition
$\fa_M=\fa_{\tL}\oplus\fa_M^L\oplus\fa_L^{1-\theta}$. 
    Then, we have $\Sigma(\tG,\tL)=\Sigma(G,M,\fa_M^L\oplus\fa_L^{1-\theta})^\sharp$.
\end{proposition}
\begin{proof}
This assertion follows from Lemmas \ref{lem:indfun1} and \ref{lem:indfun2}.
\end{proof}

\section{On the normalization of measures}\label{sec:measure}


We introduce the following notation.
\begin{itemize}
    \item Take a maximal compact subgroup $K=K_\inf K_f$ of $G(\A)=G(F_\inf)G(\A_f)$, which is admissible relative to $M_0$, see \cite{Art81}*{\S1}. 
    \item $A_0$ is the identity component of $T_0(\R)$, which is viewed as a subgroup of $T_0(\A)$ via the diagonal embedding of $\R$ into $F_\inf$. 
    \item For $M\in \cL^G(M_0)$, we set $A_M\coloneqq A_0\cap T_M(\R)$. 
    \item For $P\in\cP^G(M)$, we denote by $\delta_P$ the modular function of $P(\A)$. 
\end{itemize}

Let us fix an inner product on $\fa_{M_0}$, which is $W_0^G$-invariant and $\theta$-invariant.
This inner product endows $\fa_{M_0}$ with the structure of a Euclidean space. 
In addition, it determines Haar measures on $\fa_{M}^L$ and $\fa_{\tM}^{\tL}$.
The dual measures on $(\fa_{M}^{L})^*$ and $(\fa_{\tM}^{\tL})^*$ are then fixed accordingly.
Moreover, via the isomorphism $H_M\colon A_M\to \fa_M$, a measure on $A_M$ is also fixed.

Let us fix Haar measures on groups following \cite{Art78}*{\S1}. 
First, fix a Haar measure $dx$ on $G(\A)$. 
Normalize the Haar measure $dk$ on $K$ so that the volume of $K$ is $1$. 
For each $P\in\cF^G(M_0)$, normalize the Haar measure $dn$ on $N_P(\A)$ so that the volume of $N_P(F)\bs N_P(\A)$ is $1$. 
Note that the Iwasawa decomposition $G(\A)=N_P(\A) M_P(\A)K$ holds. 
Then normalize the Haar measure $dm$ on $M_P(\A)$ so that
\[
\int_{G(\A)} f(x)\, dx =\int_{N_P(\A)} \int_{M_P(\A)} \int_K f(nmk) \, \delta_P(m)^{-1}  dk\, dm \, dn, \qquad  f\in C_c(G(\A)) 
\]
holds.
Induce a measure $d\widetilde{x}$ on $\tG(\A)$ from the measure $dx$ on $G(\A)$ by
\[
\int_{\tG(\A)}f(\widetilde{x})\, d\widetilde{x}=\int_{G(\A)} f(x\delta) \, dx,\qquad f\in C_c(\tG(\A)).
\]

\section{Intertwining operators}\label{sec:intertwining}

Following the exposition in \cite{FLM11}*{\S2.2}, we briefly recall the intertwining operators.
For $P\in\cP^G(M)$, we introduce the following notation.
\begin{itemize}
    \item $\bar{\cA}_2(P)$ is the Hilbert space completion of 
    \[
    \{ \phi\in C^\inf(M(F)N_P(\A)\bs G(\A) ) \mid \delta_P^{-\frac12}\phi(\cdot x)\in L^2_{\mathrm{disc}}(A_M M(F)\bs M(\A)), \; \forall  x\in G(\A)  \}
    \]
    with respect to the inner product
    \[
    (\phi_1,\phi_2)\coloneqq \int_{A_M M(F) N_P(\A)\bs G(\A)} \phi_1(g)\, \overline{\phi_2(g)} \, \rd g .
    \]
    \item $\cU(\fg_\C)$ is the universal enveloping algebra of the complexified Lie algebra of $G(F_\inf)$, and $\fz$ is the center of $\cU(\fg_\C)$. 
    \item $H_P\colon G(\A)\to \fa_P$ is the extension of $H_M$ to a left $N_P(\A)$- and right $K$-invariant map.
    \item $\cA^2(P)$ is the dense subspace of $\bar{\cA}^2(P)$ consisting of its $K$- and $\fz$-finite vectors, i.e. the space of automorphic forms $\phi$ on $N_P(\A)M(F)\bs G(\A)$ such that for all $k\in K$ the function
    \[
    \delta_P^{-1/2}\phi(\cdot k)\colon M(\A)\to \C
    \]
    is a square-integrable automorphic form on $A_M M(F)\bs M(\A)$.
    \item $L^2_\disc(A_M M(F)\bs M(\A))$ is the discrete part of $L^2(A_M M(F)\bs M(\A))$. 
    \item $\rho(P,\lambda)$ $(\lambda\in \fa^*_{M,\C})$ is the induced representation of $G(\A)$ on $\bar{\cA}^2(P)$ given by
    \[
    (\rho(P,\lambda,y)\phi)(x)
    =
    \phi(xy)e^{\langle \lambda,H_P(xy)-H_P(x)\rangle};
    \]
    it is isomorphic to
    \[
    \Ind_{P(\A)}^{G(\A)}
    \left(
    L^2_{\disc}(A_M M(F)\bs M(\A))
    \otimes
    e^{\langle \lambda,H_M(\cdot)\rangle}
    \right).
    \]
\end{itemize}

For $P$, $Q\in \cP^G(M)$, let
\[
M_{Q|P}(\lambda)\colon \cA^2(P)\to \cA^2(Q),
\qquad
\lambda\in \fa^*_{M,\C},
\]
be the standard intertwining operator (cf. \cite{Art82}*{\S1}), which is the meromorphic continuation in $\lambda$ of the integral
\[
[M_{Q\mid P}(\lambda)\phi](x)
=
\int_{N_Q(\A)\cap N_P(\A)\bs N_Q(\A)}
\phi(nx)
e^{\langle \lambda,H_P(nx)-H_Q(x)\rangle}
\, dn,
\]
for $\phi\in \cA^2(P)$ and $x\in G(\A)$.
These operators satisfy the following properties:
\begin{itemize}
    \item
    $M_{P\mid P}(\lambda)\equiv \mathrm{Id}$ for all $P\in \cP^G(M)$ and $\lambda\in \fa^*_{M,\C}$.
    \item
    For any $P$, $Q$, $R\in \cP^G(M)$ we have $M_{R\mid P}(\lambda)=M_{R\mid Q}(\lambda)\circ M_{Q\mid P}(\lambda)$ for all $\lambda\in \fa^*_{M,\C}$. 
    In particular, $M_{Q\mid P}(\lambda)^{-1}=M_{P\mid Q}(\lambda)$.
    \item $M_{Q\mid P}(\lambda)^*=M_{P|Q}(-\lambda)$
    for any $P$, $Q\in \cP^G(M)$ and $\lambda\in \fa^*_{M,\C}$.
    In particular, $M_{Q\mid P}(\lambda)$ is unitary for $\lambda\in i\fa^*_M$.
    \item
    If $P\mid^\alpha Q$, then $M_{Q\mid P}(\lambda)$ depends only on
    $\langle \lambda,\alpha^\vee\rangle$.
\end{itemize}

Let $P\in \cP^G(M)$ and $\lambda\in i\fa^*_M$. 
For $Q\in \cP^G(M)$ and $\Lambda\in i\fa^*_M$, we define
\[
\cM_Q(P,\lambda,\Lambda)
\coloneqq
M_{Q|P}(\lambda)^{-1}
M_{Q|P}(\lambda+\Lambda)
=
M_{P|Q}(\lambda)
M_{Q|P}(\lambda+\Lambda).
\]
Then, $\{\cM_Q(P,\lambda,\cdot)\mid Q\in \cP^G(M)\}$ is a $(G,M)$-family with values in the space of operators on $\cA^2(P)$, see \cite{Art82}*{p.1310}.

Let $M\in\cL^G(M_0)$ be a Levi subgroup, and let $\tL\in \cL^{\tG}(\tM_0)$ be a Levi subset satisfying $L\supset M$. 
Let $\tlambda\in \fa^*_{\tL,\C}$. 
By Proposition \ref{prop:GM}, 
$\{\cM_Q(P,\tlambda,\cdot)\mid \tQ\in \cP^{\tG}(\tL)\}$ is a $(\tG,\tL)$-family.
For $\tQ\in\cP^{\tG}(\tL)$, we set
    \[
    \epsilon_{\tQ}(\tLambda)\coloneqq v_{\Delta_{\tQ}}\,  \prod_{\talpha\in\Delta_{\tQ}} \langle \tLambda,\talpha^\vee\rangle^{-1}, \qquad \tLambda\in\fa_{\tL,\C}^*.
    \]
Here, $v_{\Delta_{\tQ}}$ denotes the co-volume of the lattice generated by $\Delta_{\tQ}$ in $(\fa_{\tQ}^{\tG})^*$. 
Then, for $\tlambda$, $\tLambda\in \fa^*_{\tL,\C}$, the limit
\[
\cM_{\tL}(P,\tlambda)
\coloneqq
\lim_{\tLambda\to 0}
\sum_{\tQ\in \cP^{\tG}(\tL)}
\epsilon_{\tQ}(\tLambda)\,
\cM_Q(P,\tlambda,\tLambda)
\]
exists.

\section{The combinatorial formula}\label{sec:combiatorial}

Take $M\in\cL^G(M_0)$ and $P\in\cP^G(M)$.
We also take $\tL\in\cL^{\tG}(\tM_0)$ with $M\subset L$. 
In the following, we put
\[
m\coloneqq \dim \fa_{\tL}^{\tG}.
\]
For $\beta\in\Sigma_M\, (\subset \fa_M^*)$, write $\beta_L$ for the projection of $\beta$ onto $\fa_L^*$. 
Note that $\tbeta_L$ is the projection of $\beta_L$ onto $\fa_{\tL}^*$. 
Similarly, for $\beta^\vee\in \Sigma_M^\vee$, we denote by $\tbeta^\vee_L$ its projection onto $\fa_{\tL}$.
\begin{itemize}
    \item $\fB_{P,\tL}$ is the set of $m$-tuples $\underline{\beta}=(\beta_1^\vee,\ldots,\beta_m^\vee)$ of elements of $\Sigma_P^\vee$ such that $\tbeta_{1,L}^\vee,\dots,\tbeta_{m,L}^\vee$ form a basis of $\fa_{\tL}^{\tG}$.
\item $\vol(\underline{\beta})$ denotes the co-volume in $\fa_{\tL}^{\tG}$ of the lattice spanned by the basis $\tbeta_{1,L}^\vee,\dots,\tbeta_{m,L}^\vee$.
\item For any $\underline{\beta}=(\beta_1^\vee,\ldots,\beta_m^\vee)\in \fB_{P,\tL}$, we set
\begin{align*}
\Xi_{\tL}(\underline{\beta})
\coloneqq &
\left\{
(Q_1,\ldots,Q_m)\in \cF_1^G(M)^m \;\middle|\;
\beta_j^\vee\in \fa_M^{Q_j},
\ j=1,\ldots,m
\right\} \\
=& 
\left\{
(\overline{P_1P_1'},\ldots,\overline{P_mP_m'})
\;\middle|\;
P_j\mid^{\beta_j}P_j',
\ j=1,\ldots,m
\right\}.
\end{align*}
\item  For any smooth function $f$ on $\fa_M^*$ and $\mu\in \fa_M^*$, we denote by $D_\mu f$ the directional derivative of $f$ along $\mu$. 
\item For a pair $P_1\mid^\alpha P_2$ of adjacent parabolic subgroups in $\cP^G(M)$, we write
\[
\delta_{P_1\mid P_2}(\lambda)\coloneqq D_\varpi M_{P_1\mid P_2}(\lambda)
\colon
\cA^2(P_2)\to \cA^2(P_1),
\]
where $\varpi\in \fa_M^*$ is such that $\langle \varpi,\alpha^\vee\rangle=1$.
\item For any $m$-tuple $\cX=(Q_1,\dots,Q_m)\in\Xi_{\tL}(\underline{\beta})$ with $Q_j=\overline{P_jP_j'}$, $P_j\mid^{\beta_j}P_j'$,
we set
\begin{align*}
\Delta_{\cX}(P,\lambda)\coloneqq&    \frac{\vol(\underline{\beta})}{m!}
M_{P_1\mid P}(\lambda)^{-1}
\delta_{P_1\mid P_1'}(\lambda)
M_{P_1'\mid P_2}(\lambda) \cdots \\
& \cdots\;
\delta_{P_{m-1}\mid P_{m-1}'}(\lambda)
M_{P_{m-1}'\mid P_m}(\lambda)
\delta_{P_m\mid P_m'}(\lambda)
M_{P_m'\mid P}(\lambda).
\end{align*}
\item Take $\underline{\mu}=(\mu_1,\ldots,\mu_m)\in (\fa_M^*)^m$.
For any $\underline{\beta}=(\beta_1^\vee,\ldots,\beta_m^\vee)\in \fB_{P,\tL}$,
there exist $(Q_1,\ldots,Q_m)\in \Xi_{\tL}(\underline{\beta})$ and $\tilde{\mu}\in (\fa_{\tL}^{\tG})^*$ such that
\[
\tilde{\mu}-\mu_j\in \fa_{Q_j,+}^*
\]
for all $j=1,\ldots,m$. 
Note that $\tilde{\mu}$ is uniquely determined by $\langle \tilde{\mu},\beta_j^\vee\rangle=\langle \mu_j,\beta_j^\vee\rangle$, $j=1,\dots,m$.
\item We assume that $\underline{\mu}$ satisfies the general-position conditions in \cite{FLM11}*{pp.179--180}, with minimal dependencies taken modulo $\mathfrak a_M^L\oplus\mathfrak a_L^{1-\theta}$. 
Then, $(Q_1,\dots,Q_m)$ is uniquely determined. 
Hence, we obtain a map $\cX_{\tL,\underline{\mu}}\colon \fB_{P,\tL}\to \cF_1^G(M)^m$
with $\cX_{\tL,\underline{\mu}}(\underline{\beta})\in \Xi_{\tL}(\underline{\beta})$ for all $\underline{\beta}\in \fB_{P,\tL}$. 
\end{itemize}

The following theorem is a twisted version of \cite{FLM11}*{Theorem 2}.
\begin{theorem}\label{thm:main}
Let $M\in\cL^G(M_0)$, $P\in\cP^G(M)$, and $\tL\in\cL^{\tG}(\tM_0)$ with $M\subset L$. 
Set $m\coloneqq \dim \fa_{\tL}^{\tG}$, and suppose that $\underline{\mu}\in(\fa_M^*)^m$ is in general position. 
Then, for $\tlambda\in\fa^*_{\tL,\C}$, we have
\begin{equation*}\label{eq:FLMmain-}
\cM_{\tL}(P,\tlambda)
=
\sum_{\underline{\beta}\in \fB_{P,\tL}}
\Delta_{\cX_{\tL,\underline{\mu}}(\underline{\beta})}(P,\tlambda).
\end{equation*}
\end{theorem}
\begin{proof}
Let us first consider the case where $\tG=G$.
In this case, the theorem coincides with \cite{FLM11}*{Theorem 2}. 
As explained in \cite{FLM11}*{\S2.4}, it follows directly from the combinatorial results of \cite{FL11}.
Fix $\lambda\in\fa^*_{M,\C}$ in general position, and let $V$ be a finite-dimensional $(\fz,K)$-isotypic subspace of $\cA^2(P)$.
For each $Q\in\cP^G(M)$, let $\cA_Q$ denote the Taylor expansion at $\Lambda=0$ of the restriction of $\cM_Q(P,\lambda,\Lambda)$ to $V$.
Then $\cA=(\cA_Q)_{Q\in \cP^G(M)}$ is a compatible family in the sense of \cite{FL11}*{Definition 4.1}.
As a consequence, applying \cite{FL11}*{Theorem 8.2} to the compatible family $\cA$, the polyhedral fan $\Sigma(G,M)$, and the subspace $U=\fa_M^L$, one obtains \cite{FLM11}*{Theorem 2}.
Here we also use the fact that $\Sigma(G,M,U)^\sharp=\Sigma(G,L)$.

The twisted case is similar. 
We apply \cite{FL11}*{Theorem 7.2} to $\cA$, $\Sigma(G,M)$, and $U=\fa_M^L\oplus\fa_L^{1-\theta}$. 
By Proposition \ref{prop:indfun}, the induced fan $\Sigma(G,M,U)^\sharp$ coincides with $\Sigma(\tG,\tL)$.
Rewriting the resulting sum as in \cite{FL11}*{\S8} gives the theorem.
\end{proof}

\begin{remark}\label{rem:Parab}
    At the beginning of \cite{Par19}*{Proof of Theorem 7.2}, Parab claims that
    \[
    \cM_L(P,\tlambda)=\cM_{\tL}(P,\tlambda) , \qquad  \tlambda\in \fa_{\tL,\C}^*.
    \]
    However, this is clearly incorrect, since $\cM_L(P,\lambda)$ and $\cM_{\tL}(P,\tlambda)$ involve different orders $m$ of differentiation.
    Therefore, the argument in \cite{Par19}*{Proof of Theorem 7.2} is not valid. 
    This issue is the only error in \cite{Par19}*{Proof of Theorem 7.2}. 
    After replacing it with Theorem~\ref{thm:main} and making minor modifications, his argument for absolute convergence becomes valid.
    In what follows, we prove the absolute convergence of the spectral side by following the argument of Parab \cite{Par19}*{\S7}.    
\end{remark}


\section{The operator \texorpdfstring{$\rho(P,\lambda,\tw)$}{rho}}\label{sec:rho}

We introduce the operator $\rho(P,\lambda,\tw)$.
\begin{itemize}
\item $W_0^G\coloneqq N_{G(F)}(M_0)/M_0$ is the Weyl group of $(G,T_0)$, where $N_{H'}(H)$ is the normalizer of $H$ in $H'$. For any $w\in W_0^G$, we choose a representative $n_w\in G(F)$.
\item $W_0^{\tG}\coloneqq N_{\tG(F)}(M_0)/M_0$ is the Weyl group of $(\tG,T_0)$.
Then $W^{\tG}_0=W^G_0\rtimes\theta$, and we choose $n_{\tw}=n_w\delta$ as a representative of $\tw=w\rtimes \theta$.
\item Fix $M\in\cL^G(M_0)$.
\begin{itemize}
\item $W^{\tG}(M)\coloneqq N_{\tG(F)}(M)/M$, which can be identified with a subset of $W_0^{\tG}$.
\item $W^{\tG}(M)$ acts on $\cP^G(M)$. More precisely, for $P\in \cP^G(M)$, the action is defined by $\tw  P=\mathrm{Ad}(n_{\tw})P=n_w\theta(P)n_w^{-1}$.
\end{itemize}
\item Let $P\in \cF^G(M)$ and $\tw\in W^{\tG}(M)$. Put $Q=\tw(P)$. Let $\lambda\in\fa_{P,\C}^*$ and $y\in\tG(\A)$.
Define the map $\rho(P,\lambda,\tw,y)\colon \bar{\cA}^2(P)\to\bar{\cA}^2(Q)$ by
\[
(\rho(P,\lambda,\tw,y)\phi)(x)\coloneqq \phi(n_{\tw}^{-1}xy) \, e^{\langle \lambda,H_P(n_{\tw}^{-1}xy)\rangle -\langle \tw\lambda,H_Q(x)\rangle}, \qquad x\in G(\A).
\]
Furthermore, define the map $\sS(P,\lambda,\tw)\colon \bar{\cA}^2(P)\to\bar{\cA}^2(Q)$ by
\[
(\sS(P,\lambda,\tw)\phi)(x)\coloneqq \phi(n_{\tw}^{-1}xn_{\tw}) \, e^{\langle \lambda,H_P(n_{\tw}^{-1}xn_{\tw})\rangle -\langle \tw\lambda,H_Q(x)\rangle}, \qquad x\in G(\A).
\]
\item For a function $f$ on $\tG(\A)$, define a function $L_{\tw}f$ on $G(\A)$ by $(L_{\tw}f)(x)\coloneqq f(n_{\tw}x)$, $x\in G(\A)$.
\end{itemize}

The following facts, which follow immediately from the definitions, were observed in \cite{Par19}*{Lemma 7.1}.
\begin{itemize}
\item $\rho(P,\lambda,\tw,n_{\tw}g)=\sS(P,\lambda,\tw)\circ \rho(P,\lambda,g)$, $g\in G(\A)$.
\item For a test function $f$ on $\tG(\A)$,
\begin{equation}\label{eq:composition}
\rho(P,\lambda,\tw,f)=\sS(P,\lambda,\tw)\circ \rho(P,\lambda,L_{\tw}f)
\end{equation}
holds.
\item $\sS(P,\lambda,\tw)$ is invertible and is unitary for $\lambda\in i\fa_{P}^*$.
\end{itemize}

\section{The Fr\'echet space \texorpdfstring{$\cC(\tG(\A)^1;K_0)$}{CGK}}\label{sec:Frechet}

We introduce the following notation.
\begin{itemize}
\item $G(\A)^1\coloneqq\cap_{\chi\in X^*(G)} \mathrm{ker}|\chi|$.
\item $\tG(\A)^1\coloneqq G(\A)^1\delta$.
\item $G(F_\inf)^1\coloneqq G(F_\inf)\cap G(\A)^1$.
\item $\cU(\fg_\C^1)$ is the universal enveloping algebra of the complexified Lie algebra of $G(F_\inf)^1$, and $\fz^1$ is the center of $\cU(\fg_\C^1)$. 
\end{itemize}
Note that $G(\A)^1$ is naturally identified with $A_G\bs G(\A)$.
Since $\fz$ and $\fz^1$ induce the same algebra of operators on any space on which $A_G$ acts trivially, we shall identify their actions in what follows.

For an open compact subgroup $K_0$ of $G(\A_f)$, the quotient $K_0\bs \tG(\A)^1$ is a countable disjoint union of copies of $\tG(F_\inf)^1$, and hence is a smooth manifold.
From this point of view, we introduce $\cC(\tG(\A)^1;K_0)$ as follows.
\begin{itemize}
\item $C^\inf(\tG(\A)^1;K_0)$ is the space of smooth left-$K_0$-invariant functions on $\tG(\A)^1$. 
\item Any element $X\in\cU(\fg_\C^1)$ defines a right invariant differential operator 
\[
C^\inf(\tG(\A)^1;K_0)\ni f\mapsto X*f \in C^\inf(\tG(\A)^1;K_0).
\]
\item $\cC(\tG(\A)^1;K_0)$ is the space of $f\in C^\inf(\tG(\mathbb A)^1;K_0)$ such that $|X*f|_{L^1(\widetilde G(\mathbb A)^1)}<\infty$ for all $X\in \cU(\fg_\C^1)$. 
We endow this space with the topology defined by the seminorms
\[
f\longmapsto \|X*f\|_{L^1(\tG(\A)^1)}, \qquad X\in\cU(\fg_\C^1).
\]
\end{itemize}
The space $\cC(\tG(\A)^1;K_0)$ is a Fr\'echet space.
For the definition of a Fr\'echet space, see, for example, \cite{SW99}.
For any $\tw\in W^{\tG}_0$, the map $L_{\tw}$ induces a topological isomorphism from $\cC(\tG(\A)^1;K_0)$ to $\cC(G(\A)^1;K_0^{\tw})$, where $K_0^{\tw}\coloneqq n_{\tw}^{-1} K_0 n_{\tw}$.

\section{Absolute convergence theorem}\label{sec:absoconv}

In what follows, we fix an open compact subgroup $K_0$ of $K_f$.
Let $M\in\cL^G(M_0)$ and $P\in\cP^G(M)$, and let $\tL\in\cL^{\tG}(\tM_0)$ satisfy $M\subset L$.
We also fix an element $\tw\in W^{\tG}(M)$.
We use the following notation for norms. For details, see, for example, \cite{RS80}.
\begin{itemize}
    \item $\| \cdot\|$ denotes the operator norm on $\bar{\cA}^2(P)$.
    \item $\| \cdot\|_1$ denotes the trace norm on $\bar{\cA}^2(P)$.
    \item An operator $A$ is said to be bounded if $\|A\|<+\inf$, and an operator $B$ is said to be of trace class if $\|B\|_1<\inf$.
    \item $\cU_P(\lambda,\tw)\coloneqq M_{P\mid \tw P}(0) \circ \sS(P,\lambda,\tw)$ is a unitary operator on $\bar{\cA}^2(P)$ for $\lambda\in i(\fa_{\tL}^{\tG})^*$; in particular, $\| \cU_P(\lambda,\tw)\|=1$.
\end{itemize}

The following assertion is the twisted analogue of \cite{FLM11}*{Theorem 3}.
\begin{theorem}\label{thm:abs}
    Fix an open compact subgroup $K_0$ of $K_f$.
    Let $M\in\cL^G(M_0)$, $P\in\cP^G(M)$, and $\tL\in\cL^{\tG}(\tM_0)$ with $M\subset L$.
    Choose an element $\tw\in W^{\tG}(M)$.
    Then, for any $\underline{\beta}\in \fB_{P,\tL}$ and $\cX\in \Xi_{\tL}(\underline{\beta})$, 
    the seminorm 
    \[
    \int_{i(\fa_{\tL}^{\tG})^*} \|\Delta_{\cX}(P,\tlambda) \ M_{P\mid \tw P}(0) \ \rho(P,\tlambda,\tw,f)  \|_1\, d\tlambda
    \]
    on $f\in \cC(\tG(\A)^1;K_0)$ is continuous.
\end{theorem}
\begin{proof}
We follow the argument of \cite{FLM11}*{\S5.1}.
First, we introduce some notation.
\begin{itemize}
    \item $\Pi_\disc(M(\A))$ denotes the set of equivalence classes of irreducible unitary representations of $M(\A)$ which occur in the irreducible decomposition of $L^2_\disc(A_M\, M(F)\bs M(\A))$.
    \item $L^2_\disc(A_M M(F)\bs M(\A))$ decomposes as the Hilbert direct sum of the $\pi$-isotypic components for $\pi\in\Pi_\disc(M(\A))$. Corresponding to this decomposition, we obtain the Hilbert direct sum decomposition $\bar{\cA}^2(P)=\oplus_{\pi\in\Pi_{\disc}(M(\A))} \bar{\cA}_\pi^2(P)$.
    \item If $\cA^2_\pi(P)$ denotes the $K$-finite part of $\bar{\cA}_\pi^2(P)$, then we obtain the algebraic direct sum decomposition $\cA^2(P)=\oplus_{\pi\in\Pi_{\disc}(M(\A))} \cA_\pi^2(P)$.
    \item Let $\widehat{K_\inf}$ be the unitary dual of $K_\inf$.
    \item Let $\cA^2_\pi(P)^\tau$ denote the $\tau$-isotypic component in $\cA^2_\pi(P)$ for $\tau\in\widehat{K_\inf}$. Thus, $\cA^2_\pi(P)=\oplus_{\tau\in\widehat{K_\inf}}\cA^2_\pi(P)^\tau$ as an algebraic direct sum.
    \item Let $\cA^2_\pi(P)^{K_0,\tau}$ denote the subspace of $K_0$-invariant vectors in $\cA^2_\pi(P)^{\tau}$.
    \item Let $\Omega$ (resp. $\Omega_{K_\inf}$) be the Casimir operator of $G(F_\inf)$ (resp. $K_\inf$), and set
    \[
    \Delta=\mathrm{Id}-\Omega+2\Omega_{K_\inf}.
    \]
    \item $\rho(P,\lambda,\Delta)$ acts on $\cA^2_\pi(P)^{K_0,\tau}$ by the scalar $\mu(\pi,\lambda,\tau)\coloneqq 1+\|\lambda\|^2-\lambda_\pi+2\lambda_\tau-e_P$. Here, we denote by $\lambda_\pi$ the Casimir eigenvalue of the $M(F_\inf)$-component of $\pi\in\Pi_\disc(M(\A))$, by $\lambda_\tau$ the Casimir eigenvalue of $\tau\in \widehat{K_\inf}$, and $e_P\le 0$ is a constant depending only on $P$ (see \cite{BG83}*{\S3.2 (2)}).
    By \cite{Mul02}*{(6.9)}, one has
    \begin{equation}\label{eq:101}
        |\mu(\pi,\lambda,\tau)|^2\ge \frac{1}{4}(1+\|\lambda\|^2+\lambda_\pi^2+\lambda_\tau^2).
    \end{equation}
\end{itemize}
   By \eqref{eq:composition},
    \[
    \Delta_{\cX}(P,\tlambda) \ M_{P\mid \tw P}(0) \ \rho(P,\tlambda,\tw,f) = \Delta_{\cX}(P,\tlambda) \ \cU_P(\tlambda,\tw) \ \rho(P,\tlambda,L_{\tw} f) .
    \]
    Choose a sufficiently large $k\in\N$.
    Since $f$ is left-$K_0$-invariant, the image of $\rho(P,\tlambda,L_{\tw}f)$ is contained in the subspace of $K_0^{\tw}$-invariant vectors.  
    Hence
    \[
     \rho(P,\tlambda,L_{\tw}f)=\mathrm{pr}_{K_0^{\tw}}\ \rho(P,\tlambda,L_{\tw}f) ,
    \]
    where $\mathrm{pr}_{K_0^{\tw}}$ denotes the orthogonal projection of $\bar\cA^2(P)$ to the subspace of $K_0^{\tw}$-invariant vectors in $\bar\cA^2(P)$, that is, $\mathrm{pr}_{K_0^{\tw}}=\int_{K_0^{\tw}}\rho(P,\tlambda,k)\, dk$ with $dk$ normalized to have total mass one. This projection generally depends on $\widetilde\lambda$, but we suppress this dependence in the notation.
    Since $\mu(\pi,\lambda,\tau)$ implies $\rho(P,\tlambda,\Delta)=\rho(P,0,\Delta)+\|\tlambda\|^2\ \mathrm{Id}$, we obtain $\rho(P,\tlambda,\Delta) \ M_{P\mid \tw P}(0)=M_{P\mid \tw P}(0)\ \rho(\tw P,\widetilde w\widetilde\lambda,\Delta)$.
    Hence, 
    \[
     \rho(P,\tlambda,\Delta^{2k})\ \cU_P(\tlambda,\tw)=\cU_P(\tlambda,\tw)\ \rho(P,\tlambda,\Ad(n_{\tw}^{-1})\Delta^{2k}). 
    \]
    Using this relation, $\cU_P(\tlambda,\tw)\mathrm{pr}_{K_0^{\tw}}=\mathrm{pr}_{K_0} \cU_P(\tlambda,\tw)$, and $X*(L_{\tw} f)=L_{\tw}((\Ad(n_{\tw})X)*f)$, we obtain
    \begin{align*}
        \rho(P,\tlambda,\Delta^{2k})\ \cU_P(\tlambda,\tw) \ \rho(P,\tlambda,L_{\tw} f)=&\mathrm{pr}_{K_0} \ \cU_P(\tlambda,\tw)\ \rho(P,\tlambda,\Ad(n_{\tw}^{-1})\Delta^{2k})\ \rho(P,\tlambda,L_{\tw} f) \\
        =& \mathrm{pr}_{K_0} \ \cU_P(\tlambda,\tw)\ \rho(P,\tlambda,L_{\tw}(\Delta^{2k}* f)) .
    \end{align*}
    Recall that, for operators $A$ and $B$, one has $\|AB\|_1\le \|A\|_1\|B\|$.
    Since $\cU_P(\tlambda,\tw)$ is unitary, we have $\| \cU_P(\tlambda,\tw)\|=1$.  
    Hence, it follows from these facts that
    \begin{align*}
        & \|\Delta_{\cX}(P,\tlambda) \ \cU_P(\tlambda,\tw) \ \rho(P,\tlambda,L_{\tw} f)\|_1 \\
        &\le \|\Delta_{\cX}(P,\tlambda)   \ \rho(P,\tlambda,\Delta^{2k})^{-1} \ \mathrm{pr}_{K_0} \ \cU_P(\tlambda,\tw) \|_1 \  \|\rho(P,\tlambda,L_{\tw}(\Delta^{2k}* f))\| \\
        &\le  \|\Delta_{\cX}(P,\tlambda)  \ \rho(P,\tlambda,\Delta^{2k})^{-1}\  \mathrm{pr}_{K_0}\|_1  \ \| \Delta^{2k}* f \|_{L^1(\tG(\A)^1)}.
    \end{align*}
    Therefore, in order to prove the theorem, it suffices to show that, for sufficiently large $k$,
    \begin{equation}\label{eq:102}
        \int_{i(\fa_{\tL}^{\tG})^*} \|\Delta_{\cX}(P,\tlambda)  \ \rho(P,\tlambda,\Delta^{2k})^{-1}\  \mathrm{pr}_{K_0}\|_1 \, d\tlambda        
    \end{equation}
    converges.

    Let $m=\dim\fa_{\tL}^{\tG}$, $\cX=(Q_1,\dots,Q_m)$, $Q_j=\overline{P_jP_j'}$, and $P_j\mid^{\beta_j}P_j'$.
    Recall that, for a linear operator $A$ on a finite-dimensional vector space $W$, one has $\|A\|_1\le \|A\| \ \dim W$. 
    For $\tlambda\in i(\fa_{\tL}^{\tG})^*$, the operator $M_{Q\mid P}(\tlambda)$ is unitary and preserves the $K_\inf$-isotypic components and the $K_0$-invariant subspaces. 
    Hence, the integral in \eqref{eq:102} is bounded by a constant multiple of
    \[
    \sum_{\tau\in\widehat
    {K_\inf}}\sum_{\pi\in\Pi_{\disc}(M(\A))}\dim(\cA^2_\pi(P)^{K_0,\tau})\ \int_{i(\fa_{\tL}^{\tG})^*} |\mu(\pi,\tlambda,\tau)|^{-2k} \ \prod_{j=1}^m \| \delta_{P_j\mid P_j'}(\tlambda)|_{\cA^2_\pi(P'_j)^{K_0,\tau}}\|\ d\tlambda.
    \]
    Since $\delta_{P_j\mid P_j'}(\tlambda)$ depends only on $\langle\tlambda,\beta_j^\vee\rangle=\langle\tlambda,\tbeta_{j,L}^\vee\rangle$, if $\varpi_1,\dots,\varpi_m$ denote the dual basis of $\tbeta_{1,L}^\vee,\dots,\tbeta_{m,L}^\vee$ in $(\fa_{\tL}^{\tG})^*$, then the coordinates of $\tlambda$ with respect to $\varpi_1,\dots,\varpi_m$ are given by $\langle\tlambda,\tbeta_{j,L}^\vee\rangle$.
    Moreover, by \cite{FLM11}*{Proposition 1}, there exist constants $C_j$, $N_j$, $N'_j>0$ such that
    \[
    \int_{i \R} \| \delta_{P_j\mid P_j'}(s\varpi_j)|_{\cA^2_\pi(P_j')^{K_0,\tau}}  \| \ (1+|s|)^{-N_j} ds
    \le C_j \ (1+\lambda_\pi^2+\lambda_\tau^2)^{N_j'}.
    \]
    Therefore, applying \cite{FLM11}*{Proposition 1} to the right-hand side of the above inequality, together with the coordinates with respect to $\varpi_1,\dots,\varpi_m$ and \eqref{eq:101}, we obtain that, for sufficiently large $k$, the integral in \eqref{eq:102} is bounded by a constant multiple of
    \[
     \sum_{\tau\in\widehat
    {K_\inf}}\sum_{\pi\in\Pi_{\disc}(M(\A))}\dim(\cA^2_\pi(P)^{K_0,\tau})\ (1+\lambda_\pi^2+\lambda_\tau^2)^{-k/2}.
    \]
    The absolute convergence of the right-hand side of this inequality follows from \cite{Mul98}*{Corollary 0.3}.
\end{proof}

\section{Modified kernel}\label{sec:kernel}

Before describing the refinement of the spectral side in the next section, we briefly recall the modified kernel in the twisted setting.
In \cite{LW13} and \cite{Par19}, twisted representations with respect to a fixed unitary character of $A_G\, G(F)\bs G(\A)$ are considered.
For the sake of simplicity of exposition, however, we assume throughout this paper that the unitary character is trivial, and derive the absolute convergence of the spectral side from Theorems~\ref{thm:main} and \ref{thm:abs}.

Let $V$ be a vector space.
A twisted representation $\tpi$ of $G(\A)$ on $V$ with respect to $\tG(\A)$ is a pair consisting of a representation $\pi\colon G(\A)\to \GL(V)$ of $G(\A)$ and an invertible operator $\tpi(\delta)\in \GL(V)$ satisfying
$\tpi(x\delta y)=\pi(x)\, \tpi(\delta)\, \pi(y)$ $(x,y\in G(\A))$.
In particular, we have $\tpi(\delta)\pi(x)=\pi\circ\theta(x)\, \tpi(\delta)$.

For $\tP\in \cF^{\tG}(\tM_0)$ and $P\supset P_0$, the twisted representation $\widetilde{R}_P$ of the right regular representation $R_P$ on
$L^2(M_P(F)N_P(\A)\bs G(\A)^1)$ is defined by
\[
(\widetilde{R}_P(\delta)\phi)(x)\coloneqq \phi(\theta^{-1}( x )),\qquad \phi\in L^2(M_P(F)N_P(\A)\bs G(\A)^1)
\]
For $f\in C_c^\inf(\tG(\A)^1)$, the integral kernel $K_{\tP}(x,y)$ of the operator $\widetilde{R}_P(f)$ is given by
\[
K_{\tP}(x,y)\coloneqq \int_{N_P(\A)} \sum_{\gamma\in \tM_P(F)}f(x^{-1}\gamma ny)\ dn ,\quad x,\ y\in M_P(F)N_P(\A)\bs G(\A)^1
\]
Let $\hat\Delta_{\tP}$ be the basis of $(\fa_{\tP}^{\tG})^*$ dual to $\Delta_{\tP}^\vee$, and let $\hat\tau_{\tP}$ be the characteristic function of $\{H\in\fa_{P_0}^{G}\mid \widetilde{\varpi}(H)>0$ for all $ \widetilde{\varpi}\in \hat\Delta_{\tP} \}$.
Then, for a truncation parameter $T\in\fa_{P_0}^{G}$,
\[
J^{T,\tG}(f)\coloneqq \int_{G(F)\bs G(\A)^1} k^T(x) \ dx ,\qquad f\in C_c^\inf(\tG(\A)^1),
\]
\[
k^T(x)\coloneqq \sum_{\substack{ \tP\in\cF^{\tG}(\tM_0), \\ P\supset P_0}} (-1)^{\dim\fa_{\tP}^{\tG}} \sum_{\xi\in P(F)\bs G(F)} K_{\tP}(\xi x,\xi x)\ \hat\tau_{\tP}(H_{P_0}(\xi x)-T)
\]
define the integral $J^{T,\tG}(f)$ of the modified kernel $k^T(x)$.
By \cite{LW13}*{Th\'eor\`eme 9.1.2}, the integral $J^{T,\tG}(f)$ converges absolutely in the region where $\min_{\alpha\in\Delta_{P_0}}\alpha(T)$ is sufficiently large.
Moreover, as in \cite{LW13}*{Th\'eor\`eme 11.1.1}, the integral $J^{T,\tG}(f)$ is a polynomial in $T$.
Let $T_0$ be the point of $\fa_{P_0}^{G}$ defined in \cite{LW13}*{Lemme 3.3.3}.
We denote by $J^{\tG}(f)$ the value of $J^{T,\tG}(f)$ at $T=T_0$.
The distribution $J^{\tG}(f)$ is independent of the choice of $\tP_0\in\cP^{\tG}(\tM_0)$.
Although the notation $T_0$ for this point conflicts with the split torus $T_0$ fixed in \S\ref{sec:setup} of this paper, we follow the standard convention and use the same notation $T_0$.

\section{A refinement of the spectral side}\label{sec:refinement}

Finally, we describe the refinement of the spectral side of $J^{\tG}(f)$ given in \cite{LW13}*{\S14.3}.
We then present a formulation combined with the combinatorial formula, Theorem~\ref{thm:main}.

We fix the following notation.
\begin{itemize}
    \item $\tW^L_0\coloneqq N_{L(F)}(T_{\tM_0})/M_0$. 
    \item $\theta_L$ denotes the restriction of the action of $\theta$ to $\fa_L$.
    \item $W^{\tG}(M)_\reg\coloneqq\{ \tw\in W^{\tG}(M) \mid \det( (\tw-1)|_{\fa_M^G})\neq 0  \}$.
    \item For $\tL\in\cL^{\tG}(\tM_0)$, $M\in\cL^L(M_0)$, and $\tw\in W^{\tL}(M)_\reg$, set
    \[
    \iota_{\tL,M,\tw}\coloneqq \frac{|\tW^L_0|}{|\tW^G_0|}\frac{1}{|\det((\theta_L-1)|_{\fa_L^G/\fa_{\tL}^{\tG}})|} \   \frac{|W^M_0|}{|W^L_0|}\frac{1}{|\det((\tw-1)|_{\fa_M^L})|}.
    \]
\end{itemize}
The theorem \cite{LW13}*{Th\'eor\`eme 14.3.1} was proved by using the formula for each cuspidal datum, \cite{LW13}*{Th\'eor\`eme 14.2.1}, together with the absolute convergence theorem, Theorem~\ref{thm:abs}.
As a result, in \cite{LW13}*{Th\'eor\`eme 14.3.3}, as a refinement of the spectral side of $J^{\tG}(f)$, $f\in C_c^\inf(\tG(\A)^1)$, the distribution $J^{\tG} (f)$ is expanded as
\[
\sum_{\tL\in\cL^{\tG}(\tM_0)} \sum_{M\in\cL^L(M_0)} \sum_{\tw\in W^{\tL}(M)_\reg}\iota_{\tL,M,\tw} \ \int_{i(\fa_{\tL}^{\tG})^*} \tr\left( \cM_{\tL}(P,\tlambda) \ M_{P\mid\tw P}(0)\ \rho(P,\tlambda,\tw,f) \right) \ d\tlambda.
\]
Here, for each $M$, we choose an arbitrary element $P\in\cP^G(M)$.

Let $\cC(\tG(\A)^1)$ denote the inductive limit of $\cC(\tG(\A)^1;K_0)$ over all open compact subgroups $K_0$ of $K_f$.
By Theorems~\ref{thm:main} and \ref{thm:abs} and the spectral expansion above, the distribution $J^{\tG}$, initially defined on $C_c^\infty(\tG(\mathbb A)^1)$, therefore extends uniquely to a continuous linear functional on $\cC(\tG(\A)^1)$. We use the same notation for this extension.
\begin{theorem}\label{thm:conti}
For every $f\in\cC(\tG(\A)^1)$, the continuous extension $J^{\tG}(f)$ is given by
\begin{align*}
 J^{\tG}(f)=&\sum_{\tL\in\cL^{\tG}(\tM_0)} \sum_{M\in\cL^L(M_0)} \sum_{\tw\in W^{\tL}(M)_\reg}\iota_{\tL,M,\tw} \\ 
& \sum_{\underline{\beta}\in \fB_{P,\tL}}
\int_{i(\fa_{\tL}^{\tG})^*} \tr\left( \Delta_{\cX_{\tL,\underline{\mu}}(\underline{\beta})}(P,\tlambda) \ M_{P\mid\tw P}(0)\ \rho(P,\tlambda,\tw,f) \right) \ d\tlambda .
\end{align*}
In this expression, the sums on the right-hand side are finite, and for every $f\in\cC(\tG(\A)^1)$ each integral on the right-hand side converges absolutely with respect to the trace norm.
Hence, each integral on the right-hand side, and hence $J^{\tG}(f)$ itself, defines a distribution on $\cC(\tG(\A)^1)$.
\end{theorem}

\bibliography{bibliography}

\begin{bibdiv}
\begin{biblist}

\bib{Art05}{incollection}{
      author={Arthur, James},
       title={An introduction to the trace formula},
        date={2005},
   booktitle={Harmonic analysis, the trace formula, and {S}himura varieties},
      series={Clay Math. Proc.},
      volume={4},
   publisher={Amer. Math. Soc., Providence, RI},
       pages={1\ndash 263},
      review={\MR{2192011}},
}

\bib{Art78}{article}{
      author={Arthur, James~G.},
       title={A trace formula for reductive groups. {I}. {T}erms associated to classes in {$G({\bf Q})$}},
        date={1978},
        ISSN={0012-7094,1547-7398},
     journal={Duke Math. J.},
      volume={45},
      number={4},
       pages={911\ndash 952},
         url={http://projecteuclid.org/euclid.dmj/1077313104},
      review={\MR{518111}},
}

\bib{Art81}{article}{
      author={Arthur, James},
       title={The trace formula in invariant form},
        date={1981},
        ISSN={0003-486X},
     journal={Ann. of Math. (2)},
      volume={114},
      number={1},
       pages={1\ndash 74},
         url={https://doi.org/10.2307/1971376},
      review={\MR{625344}},
}

\bib{Art82}{article}{
      author={Arthur, James},
       title={On a family of distributions obtained from {E}isenstein series. {II}. {E}xplicit formulas},
        date={1982},
        ISSN={0002-9327,1080-6377},
     journal={Amer. J. Math.},
      volume={104},
      number={6},
       pages={1289\ndash 1336},
         url={https://doi.org/10.2307/2374062},
      review={\MR{681738}},
}

\bib{Art88b}{article}{
      author={Arthur, James},
       title={The invariant trace formula. {II}. {G}lobal theory},
        date={1988},
        ISSN={0894-0347,1088-6834},
     journal={J. Amer. Math. Soc.},
      volume={1},
      number={3},
       pages={501\ndash 554},
         url={https://doi.org/10.2307/1990948},
      review={\MR{939691}},
}

\bib{BG83}{article}{
      author={Borel, Armand},
      author={Garland, H.},
       title={Laplacian and the discrete spectrum of an arithmetic group},
    language={English},
        date={1983},
        ISSN={0002-9327},
     journal={Am. J. Math.},
      volume={105},
       pages={309\ndash 335},
}

\bib{Bor91}{book}{
      author={Borel, Armand},
       title={Linear algebraic groups},
     edition={Second},
      series={Graduate Texts in Mathematics},
   publisher={Springer-Verlag, New York},
        date={1991},
      volume={126},
        ISBN={0-387-97370-2},
         url={https://doi.org/10.1007/978-1-4612-0941-6},
      review={\MR{1102012}},
}

\bib{FL11}{article}{
      author={Finis, Tobias},
      author={Lapid, Erez},
       title={On the spectral side of {A}rthur's trace formula---combinatorial setup},
        date={2011},
        ISSN={0003-486X,1939-8980},
     journal={Ann. of Math. (2)},
      volume={174},
      number={1},
       pages={197\ndash 223},
         url={https://doi.org/10.4007/annals.2011.174.1.6},
      review={\MR{2811598}},
}

\bib{FLM11}{article}{
      author={Finis, Tobias},
      author={Lapid, Erez},
      author={M\"uller, Werner},
       title={On the spectral side of {A}rthur's trace formula---absolute convergence},
        date={2011},
        ISSN={0003-486X,1939-8980},
     journal={Ann. of Math. (2)},
      volume={174},
      number={1},
       pages={173\ndash 195},
         url={https://doi.org/10.4007/annals.2011.174.1.5},
      review={\MR{2811597}},
}

\bib{FLM15}{article}{
      author={Finis, Tobias},
      author={Lapid, Erez},
      author={M\"{u}ller, Werner},
       title={Limit multiplicities for principal congruence subgroups of {${\rm GL}(n)$} and {${\rm SL}(n)$}},
        date={2015},
        ISSN={1474-7480,1475-3030},
     journal={J. Inst. Math. Jussieu},
      volume={14},
      number={3},
       pages={589\ndash 638},
         url={https://doi.org/10.1017/S1474748014000103},
      review={\MR{3352530}},
}

\bib{FM21}{article}{
      author={Finis, Tobias},
      author={Matz, Jasmin},
       title={On the asymptotics of {H}ecke operators for reductive groups},
        date={2021},
        ISSN={0025-5831,1432-1807},
     journal={Math. Ann.},
      volume={380},
      number={3-4},
       pages={1037\ndash 1104},
         url={https://doi.org/10.1007/s00208-021-02163-0},
      review={\MR{4297181}},
}

\bib{LW13}{book}{
      author={Labesse, Jean-Pierre},
      author={Waldspurger, Jean-Loup},
       title={La formule des traces tordue d'apr\`es le {F}riday {M}orning {S}eminar},
      series={CRM Monograph Series},
   publisher={American Mathematical Society, Providence, RI},
        date={2013},
      volume={31},
        ISBN={978-0-8218-9441-5},
         url={https://doi.org/10.1090/crmm/031},
        note={With a foreword by Robert Langlands [dual English/French text]},
      review={\MR{3026269}},
}

\bib{Mul02}{article}{
      author={M\"uller, W.},
       title={On the spectral side of the {A}rthur trace formula},
        date={2002},
        ISSN={1016-443X,1420-8970},
     journal={Geom. Funct. Anal.},
      volume={12},
      number={4},
       pages={669\ndash 722},
         url={https://doi.org/10.1007/s00039-002-8263-7},
      review={\MR{1935546}},
}

\bib{Mul98}{article}{
      author={M{\"u}ller, W.},
       title={The trace class conjecture in the theory of automorphic forms. {II}},
    language={English},
        date={1998},
        ISSN={1016-443X},
     journal={Geom. Funct. Anal.},
      volume={8},
      number={2},
       pages={315\ndash 355},
}

\bib{Par19}{article}{
      author={Parab, Abhishek},
       title={Absolute convergence of the twisted trace formula},
        date={2019},
        ISSN={0025-5874,1432-1823},
     journal={Math. Z.},
      volume={292},
      number={1-2},
       pages={529\ndash 567},
         url={https://doi.org/10.1007/s00209-019-02290-0},
      review={\MR{3968914}},
}

\bib{RS80}{book}{
      author={Reed, Michael},
      author={Simon, Barry},
       title={Methods of modern mathematical physics. {I}},
     edition={Second},
   publisher={Academic Press, Inc. [Harcourt Brace Jovanovich, Publishers], New York},
        date={1980},
        ISBN={0-12-585050-6},
        note={Functional analysis},
      review={\MR{751959}},
}

\bib{Spr98}{book}{
      author={Springer, T.~A.},
       title={Linear algebraic groups},
     edition={Second},
      series={Progress in Mathematics},
   publisher={Birkh\"auser Boston, Inc., Boston, MA},
        date={1998},
      volume={9},
        ISBN={0-8176-4021-5},
         url={https://doi.org/10.1007/978-0-8176-4840-4},
      review={\MR{1642713}},
}

\bib{SW99}{book}{
      author={Schaefer, H.~H.},
      author={Wolff, M.~P.},
       title={Topological vector spaces},
     edition={Second},
      series={Graduate Texts in Mathematics},
   publisher={Springer-Verlag, New York},
        date={1999},
      volume={3},
        ISBN={0-387-98726-6},
         url={https://doi.org/10.1007/978-1-4612-1468-7},
      review={\MR{1741419}},
}

\bib{TW26}{article}{
      author={Takanashi, Yugo},
      author={Wakatsuki, Satoshi},
       title={Asymptotic behavior for twisted traces of self-dual and conjugate self-dual representations of {${\rm GL}_n$}},
        date={2026},
        ISSN={0001-8708,1090-2082},
     journal={Adv. Math.},
      volume={493},
       pages={Paper No. 110924, 99},
         url={https://doi.org/10.1016/j.aim.2026.110924},
      review={\MR{5050433}},
}

\bib{Zie95}{book}{
      author={Ziegler, G\"unter~M.},
       title={Lectures on polytopes},
      series={Graduate Texts in Mathematics},
   publisher={Springer-Verlag, New York},
        date={1995},
      volume={152},
        ISBN={0-387-94365-X},
         url={https://doi.org/10.1007/978-1-4613-8431-1},
      review={\MR{1311028}},
}

\end{biblist}
\end{bibdiv}
\bibliographystyle{amsplain}

\end{document}